\documentclass[10pt]{amsart}
\usepackage{amssymb,amsmath}
\usepackage{enumerate}

\newtheorem{theorem}{Theorem}
\newtheorem{lemma}[theorem]{Lemma}
\newtheorem{proposition}[theorem]{Proposition}
\newtheorem{definition}[theorem]{Definition}

\newtheorem{remark}[theorem]{Remark}
\newtheorem{remarks}[theorem]{Remarks}
\newtheorem{corollary}[theorem]{Corollary}

\begin{document}
\title{A pointwise bipolar theorem}
\author{Daniel Bartl
\and Michael Kupper} 
\thanks{Department of Mathematics, University of Konstanz, daniel.bartl@uni-konstanz.de
	\and Department of Mathematics, University of Konstanz, kupper@uni-konstanz.de.}
\date{\today}

\begin{abstract} 
We provide a pointwise bipolar theorem for $\liminf$-closed convex sets of positive Borel measurable functions on a $\sigma$-compact metric space without the assumption that the polar is a tight set of measures.
As applications we derive a version of the transport duality under non-tight marginals, and a superhedging duality for semistatic hedging in discrete time.
\end{abstract}

\maketitle
\setcounter{equation}{0}

\section{Introduction}	

Given a dual pair of vector spaces $(X,Y,\langle\cdot,\cdot\rangle)$, the bipolar theorem states that every $\sigma(X,Y)$-closed, convex set $A$ with $0\in A$ is equal to its bipolar $A^{\circ\circ}$, where we recall $A^\circ=\{y\in Y: \langle x,y\rangle\leq 1\text{ for all }x\in A\}$ and $A^{\circ\circ}=\{x\in X: \langle x,y\rangle\leq 1\text{ for all }y\in A^\circ\}$.
The result is a straightforward application of the Hahn-Banach separation theorem for locally convex topological vector spaces. Motivated by applications in mathematical finance, Brannath and Schachermayer \cite{brannath1999bipolar} provide a version of the bipolar theorem on the cone $L^0_+:=L^0_+(\Omega,\mathcal{F},P)$ endowed with the topology induced by convergence in probability which is not locally convex: 
A monotone (see the definition below) nonempty set $A\subset L^0_+$ is convex and closed w.r.t.~convergence in probability if and only if $A=A^{\circ\circ}$. Here 
$A^\circ=\{ g\in L^0_+ : E_P[fg] \leq 1 \text{ for all } f\in A\}$ 
and $A^{\circ\circ}:=\{ f\in L^0_+ : E_P[fg]\leq 1 \text{ for all } g\in A^\circ\}$.	

In the present work we focus on a pointwise version of the bipolar theorem of Brannath and Schachermayer 
when functions that are almost surely equal w.r.t.~a reference measure are not identified.
To that end, let $\mathcal{L}_+^0=\mathcal{L}_+^0(\Omega)$ be the set of all Borel measurable functions
$f\colon\Omega\to[0,+\infty]$, where  $\Omega$ is a $\sigma$-compact metric space. Denoting by $ca_+$ 
the set of all finite positive Borel measures on $\Omega$, we define
$\langle f,\mu\rangle:=\int f\,d\mu$
for all measurable $f$ which are bounded form below and $\mu\in ca_+$. The polar and bipolar of a subset 
$H\subset \mathcal{L}_{+}^0$ are given by
\begin{align*}
H^\circ&:=\left\{\mu\in ca_+: \langle f,\mu\rangle\le 1\mbox{ for all }f\in H\right\},\\
H^{\circ\circ}&:=\left\{f \in \mathcal{L}_{+}^0: \langle f,\mu\rangle\le 1\mbox{ for all }\mu\in H^\circ\right\}.
\end{align*}
A set $H\subset \mathcal{L}^0_+$ is called monotone, if  $f\in H$ for all $f\in \mathcal{L}_+^0$ such that $f\leq h$ for some $h\in H$. Further, $H$ is called $\liminf$-closed whenever $\liminf_n h_n\in H$ for every sequence $(h_n)$ in $H$, and regular if $\sup_{h\in H\cap U_b}\langle h,\mu\rangle=\sup_{h\in H\cap C_b}\langle h,\mu\rangle$ for all $\mu\in ca_+$, or equivalently, if $(H\cap C_b)^\circ=(H\cap U_b)^\circ$. Here $C_b$ and $U_b$ denote the spaces of all bounded functions $f:\Omega\to\mathbb{R}$ that are continuous and upper semicontinuous, respectively. Then the following pointwise version of the bipolar theorem on $\mathcal{L}^0_+$ holds:

\begin{theorem}
\label{thm:main}
	Let $H$ be a nonempty monotone regular subset of $\mathcal{L}_{+}^0$. 
	Then $H=H^{\circ\circ}$ if and only if $H$ is convex and closed under $\liminf$.	
\end{theorem}
As an application we deduce the following regularity result:

\begin{corollary}
\label{cor:main}
	Under the assumptions of Theorem \ref{thm:main} one has $H^\circ=(H\cap C_b)^\circ$.
\end{corollary}

The following result is a reformulation of Theorem \ref{thm:main} on the level of functionals. 

\begin{corollary}
\label{cor:functional}
	Let $\phi\colon \mathcal{L}^0_+\to[0,+\infty]$ be a  convex increasing functional such that
	$\{ \phi\leq c \}$ is nonempty and regular for every $c\in\mathbb{R}$.
	For $\mu\in ca_+$ define 
	$\phi^\ast(\mu):=\sup_{f\in\mathcal{L}_+^0}(\langle f,\mu\rangle-\phi(f))$ (with the convention $+\infty-\infty:=-\infty$).
	Then
	\[ \phi(f)=\sup_{\mu\in ca_+} ( \langle f,\mu\rangle -\phi^\ast(\mu))\quad\text{for all }f\in\mathcal{L}_+^0 \]
	if and only if $\phi(f)\leq\liminf_n\phi(f_n)$ for every 
	sequence $(f_n)$ in $\mathcal{L}^0_+$ such that $f_n\to f$ pointwise to some $f\in\mathcal{L}^0_+$.
\end{corollary}

The proof of Theorem \ref{thm:main} is divided into two steps. We first present a bipolar theorem under an additional tightness assumption for  $\liminf$-closed convex sets $G$ of measurable functions that are bounded from below. In that case, it follows from the Choquet capacitability theorem (see e.g.~\cite{bartl2017robust,choquet1959forme}) 
that the ``superhegding'' functional \[\phi(f):=\inf\{ m\in\mathbb{R} : m+g\geq f \text{ for some }g\in G\}\] has a dual representation of the form $\phi(f)=\sup_{\mu\in ca_+}(\langle f,\mu\rangle-\phi^\ast(\mu))$, and the level sets $\{\phi^\ast\leq c\}$ are tight for all $c\in\mathbb{R}$, where $\phi^\ast$ denotes the convex conjugate. That $G$ is equal to 
its bipolar $G^{\circ\circ}$ then follows from the representation of the superhedging functional. In a second step, 
we approximate $H\subset \mathcal{L}^0_+$ by a sequence of $\liminf$-closed convex sets $(H_k)$ which satisfy the tightness assumption and thus by the first step
have the bipolar representation $H_k=H_k^{\circ\circ}$. 
The $\liminf$-closedness is then used to show that 
$H=\bigcap_k H_k=(\bigcup_k H^\circ_k)^\circ=H^{\circ\circ}$. 

In particular, Theorem \ref{thm:main} implies that $H\cap \mathcal{L}^\infty$ is $\sigma(\mathcal{L}^\infty,ca)$-closed, where $\mathcal{L}^\infty$ is the set of all bounded Borel measurable functions $f:\Omega\to\mathbb{R}$. In case that functions in $\mathcal{L}^\infty$ are identified if they are equal almost surely  w.r.t.~a reference measure, it follows from the Krein-Smulian theorem that a convex set in $\mathcal{L}^\infty$ is weak$^\ast$-closed if it is Fatou closed, i.e.~closed under bounded almost surely convergent sequences. If the dominating measure is replaced by a capacity, a similar result is shown in \cite{maggis2016fatou}, however under the assumption that the capacity allows for an essential infimum.

Finally, we give two applications of the pointwise bipolar theorem. The first one is a transport duality with non-tight marginals.
In the classical transport problem one optimizes $\langle f, \mu\rangle$ for a given function $f\in\mathcal{L}^0_+$ over the set of all measures
$\mu$ with prescribed marginals. Motivated by the hedging problem in mathematical finance, we consider the modified version
where $\langle f,\mu\rangle$ is optimized over all measures where the marginals are in given non-tight sets $H^\circ_i$.
By means of Theorem \ref{thm:main} we identify the modified transport problem with a corresponding superhedging functional. As a second application, we consider the problem of pointwise superreplicating a path-dependent contingent claim $f$ by investing dynamically and 
statically at the terminal time, i.e.~minimizing the hedging costs $\varphi(g)$ over the trading strategies $(\vartheta,g)$ such that 
$f(S_1,\dots,S_T)\leq (\vartheta\cdot S)_T+g(S_T)$.
Here $(\vartheta\cdot S)_T$ denotes the discrete time stochastic integral and 
$\varphi$ is a (sublinear) pricing functional for the plain vanilla option $g(S_T)$.
The bipolar theorem is then used to show the superhedging duality.
This is a classical problem in mathematical finance and was investigated e.g.~in
\cite{acciaio2017space,dolinsky2014martingale,galichon2014stochastic,hobson2011skorokhod}
though in a different setting, i.e.~either in continuous time or 
under the assumption that a reference measure exists.

The paper is organized as follows. In Section \ref{sec2} and Section \ref{sec:proof} we state and prove our main results. 
Their applications to the transport problem and the robust hedging problem are given in 
Section \ref{sec:app1} and Section \ref{sec:app2}, respectively.

\section{A bipolar theorem for $\liminf$-closed sets}\label{sec2}

Let $\Omega$ be a metric space. Denote by $\mathcal{L}_{b-}^0$ the space of all Borel measurable functions
$f\colon\Omega\to\mathbb{R}\cup\{+\infty\}$ which are bounded from below.
Let $ca_+$ be the set of all finite positive Borel measures on $\Omega$, including the subset $ca_+^1$ of all probability measures.
Define
$\langle f,\mu\rangle:=\int f\,d\mu$
for all $f\in \mathcal{L}_{b-}^0$ and $\mu\in ca_+$. The polar and bipolar sets of 
$H\subset \mathcal{L}_{b-}^0$ are given by
\[H^\circ:=\left\{\mu\in ca_+: \langle f,\mu\rangle\le 1\mbox{ for all }f\in H\right\}\]
and 
\[H^{\circ\circ}:=\left\{f \in \mathcal{L}_{b-}^0: \langle f,\mu\rangle\le 1\mbox{ for all }\mu\in H^\circ\right\}.\]
Let $C_b$ and $U_b$ be the sets of all bounded functions $f:\Omega\to\mathbb{R}$ that are continuous and upper semicontinuous, respectively. 

\begin{definition}
	We say that a subset $H$ of $\mathcal{L}_{b-}^0$ is
	\begin{enumerate}[ $\bullet$]
    \item monotone, if  $f\in H$ for all $f\in \mathcal{L}_{b-}^0$ such that $f\leq h$ for some $h\in H$,
	\item nontrivial, if $H\neq\emptyset$ and $H\neq \mathcal{L}_{b-}^0$,
	\item normalized, if $0\in H$ and $\varepsilon\notin H$ for every $\varepsilon>0$,
	\item tight, if for every $m\in\mathbb{R}$ with $m\in H$, $n\in\mathbb{N}$ and $\varepsilon>0$,
	there exists a compact set $K\subset \Omega$ such that $m-\varepsilon + n1_{K^c}\in H$,
	\item closed under $\liminf$, if 
	$\liminf_ n h_n\in H$ for every sequence $(h_n)$ in $H$ with $h_n\geq c$ for some $c\in\mathbb{R}$,
	\item regular, if $\sup_{h\in H\cap U_b}\langle h,\mu\rangle=\sup_{h\in H\cap C_b}\langle h,\mu\rangle$ for all $\mu\in ca_+^1$.
	\end{enumerate}
\end{definition}

For a normalized and monotone set $H\subset\mathcal{L}_{b-}^0$ it suffices
to restrict to $m=0$ in the definition of tightness. 
Further, if $\Omega$ is compact, every monotone set $H\subset\mathcal{L}_{b-}^0$ is automatically tight.

\begin{lemma}
\label{lem:centered}
	Suppose that $H$ is monotone and closed under $\liminf$.
	Then $H$ is nontrivial if and only if $H-m$ is normalized for a unique $m\in\mathbb{R}$.
\end{lemma}
\begin{proof}
	Obviously, if $H-m$ is normalized for some $m$, then $H$ is nontrivial.
	Conversely, if $H$ is nontrivial, define
	\[ M:=\{m\in\mathbb{R} : h\geq m \text{ for some } h\in H \}.\]
	Then $M\neq \emptyset$ and monotonicity of $H$ implies that $M$ is bounded from above, i.e. $m:=\sup M\in\mathbb{R}$.
	Let $(m_n)$ in $M$ such that $m_n\uparrow m$.
	By definition, there exists $(h_n)$ in $H$ such that $h_n\geq m_n$. 
	Since $h_n\geq m_1$ and $H$ is closed under $\mathop{\rm lim\,inf}$, 
	one has $m\leq \liminf_n h_n\in H$. Hence $m\in H$, which shows that $H-m$ is normalized.
\end{proof}

\begin{remark}
\label{rem:fatou.continuous.below}
	Every monotone subset $H$ of $\mathcal{L}_{b-}^0$ is closed under $\liminf$ if and only if
	$\sup_n h_n\in H$ for every increasing sequence $(h_n)$ in $H$. 
	Indeed, if $(h_n)$ is a sequence in $H$ which is bounded from below by a constant, then
	$\liminf h_n=\sup_n g_n$ for $g_n:=\inf_{m\geq n} h_m$ which is an element of $H$ by monotonicity.
\end{remark}

\begin{proposition}
	\label{thm:main.tight.liminf}
	Let $H$ be a monotone normalized regular tight subset of $\mathcal{L}_{b-}^0$. 
	Then,  $H=H^{\circ\circ}$ if and only if $H$ is convex and closed under $\liminf$.	
\end{proposition}
\begin{proof}
	If $H=H^{\circ\circ}$ then $H$ is convex and closed under $\liminf$ by Fatou's lemma.
	Conversely, suppose $H$ is convex and closed under $\liminf$. Define 
	\[ \phi(f):=\inf\{ m\in\mathbb{R} : m+h\geq f \text{ for some }h\in H\} \]
    for all $f\in \mathcal{L}_{b-}^0$. Then, one has $\phi(f+m)=\phi(f)+m$ for every $f\in \mathcal{L}_{b-}^0$ 
    and $m\in\mathbb{R}$, and $\phi(0)=0$ by normalization of $H$. 
	Moreover, $\phi$ is increasing and convex, by monotonicity and convexity of $H$.
    Our goal is to apply the Choquet capacitability theorem in the form of \cite[Section 2]{bartl2017robust},
    which requires ``downwards continuity '' of $\phi$ on the lattice of continuous bounded functions,
    and ``upwards continuity'' on $\mathcal{L}^0_+$.
    So let $(f_n)$ be a sequence in $C_b$ which decreases pointwise to 0.
	By tightness of $H$, for every $\varepsilon>0$ there exists
	a compact $K\subset\Omega$ such that $\|f_1\|_\infty 1_{K^c}-\varepsilon\in H$,
	where $\|\cdot\|_\infty$ denotes the supremum norm.
	It follows from Dini's lemma that $f_n1_K\leq\varepsilon$ for $n$ large enough, so that 
	$f_n\leq 2\varepsilon+\|f_1\|_\infty 1_{K^c}-\varepsilon$, and therefore $\phi(f_n)\leq2\varepsilon$.
	As $\varepsilon>0$ was arbitrary it follows that $\lim_n \phi(f_n)=0$. 
	Define
	\[ \phi^\ast_C(\mu):=\sup_{f\in C_b} \big( \langle f,\mu\rangle -\phi(f)\big)
	\quad\text{and}\quad
	\phi^\ast_U(\mu):=\sup_{f\in U_b} \big( \langle f,\mu\rangle -\phi(f)\big)\]
	for all $\mu\in ca_+$.
	Observe that $\phi^\ast_C(\mu)=\phi^\ast_U(\mu)=+\infty$ whenever $\mu$ is not a probability
	measure, because $\phi(m)=m$ for each $m\in\mathbb{R}$.
	Further, one has
	\[ \phi^\ast_C(\mu)=\sup_{f\in H\cap C_b} \langle f,\mu\rangle\]
    for every $\mu\in ca_+^1$. Indeed, by definition the left hand side is larger then the right hand side.
	To show the other inequality, fix $f\in C_b$ and $\varepsilon>0$. 
	By definition of $\phi$ there exists $h\in H$ such that 
	$\phi(f)+\varepsilon+h\geq f$. For $f':=f-\phi(f)-\varepsilon$, one has $f'\in H\cap C_b$ because $f'\le h$, and 
	\[ \langle f',\mu\rangle 
	= \langle f,\mu\rangle -\phi(f)-\varepsilon.\]
	As $\varepsilon>0$ was arbitrary, the statement holds. Similarly, it follows that 
	$\phi^\ast_U(\mu)=\sup_{f\in H\cap U_b} \langle f,\mu\rangle$, so that by regularity 
	\[ \phi^\ast_C(\mu)
	=\sup_{f\in H\cap C_b}\langle f,\mu \rangle
	=\sup_{f\in H\cap U_b}\langle f,\mu \rangle
	=\phi^\ast_U(\mu). \]
	We next show that $\phi$ is continuous from below on $\mathcal{L}_{b-}^0$. Let $(f_n)$ be a sequence 
	in $\mathcal{L}_{b-}^0$ which increases pointwise to $f\in \mathcal{L}_{b-}^0$.
	Since $\phi$ is increasing, one has
	$\phi(f)\geq\lim_n \phi(f_n)$.
	As for the other inequality, we assume that $\lim_n \phi(f_n)<\infty$ since otherwise the statement is obvious.
	For each $n$, fix $m_n\in\mathbb{R}$ and $h_n\in H$ such that
	\[ m_n\leq\phi(f_n)+1/n
	\quad \text{and}\quad
	m_n+h_n \geq f_n.\]
	Note that the sequence $(m_n)$ has a limit.
	Since $h_n\geq f_n-m_n\ge c$ for some $c\in\mathbb{R}$ and $H$ is closed under 
	$\liminf$, it follows that $h:=\liminf_n h_n\in H$.
	Hence
	\[\lim_n m_n + h 
	=\liminf_n (m_n+h_n)
	\geq\liminf_n f_n
	=f \]
	which shows that $\phi(f)\leq\lim_n\phi(f_n)$.
	Moreover, we obtain $\phi(f)\leq 0$ if and only if $f\leq h$ for some 
	$h\in H$ by applying this argument to the constant sequence $f_n:=f$ 
	for all $n\in\mathbb{N}$ which, by monotonicity, shows that $H=\{f\in \mathcal{L}^0_{b-}: \phi(f)\le 0\}$.
	Now Theorem 2.2 and Proposition 2.3 in \cite{bartl2017robust} yield 
	\begin{equation}
	\label{rep}
	\phi(f)=\sup_{\mu\in ca_+^1} \big( \langle f,\mu\rangle -\phi^\ast_C(\mu)\big)
	\end{equation}
	for all bounded $f\in\mathcal{L}^\infty:=\{f\colon\Omega\to\mathbb{R}: f\text{ is Borel and bounded}\}$.
	Note that, as stated in the beginning of Section 2 in \cite{bartl2017robust}, 
	both results are valid under the usual ZFC-axioms and do not require Martins axiom.
	For arbitrary $f\in\mathcal{L}_{b-}^0$ consider $f\wedge n$ which increases pointwise to $f$.
	Since \eqref{rep} holds for every $n$, and $\phi$ as well as $\langle\cdot,\mu\rangle$ are continuous from below,
	\eqref{rep} extends to $f\in\mathcal{L}_{b-}^0$ (with the convention $+\infty-\infty:=-\infty$ on the right hand side).
	
	Finally we show that $H=H^{\circ\circ}$. Obviously, $H\subset H^{\circ\circ}$. As for the other inclusion, fix $f\not\in H$.
	Since $\phi(f)>0$ it follows from \eqref{rep} that there exists $\mu\in ca_+^1$ such that
	\begin{equation}\label{sep}
	\langle f,\mu\rangle > \phi^*_C(\mu)\ge 0.
	\end{equation}
	Further, since $\phi(h)\leq 0$ for every $h\in H$, it follows again from \eqref{rep} 
	that $\phi^\ast_C(\mu)\geq \langle h,\mu\rangle$ for every $h \in H$.
	Hence, by scaling \eqref{sep} there exists $\mu'\in ca_+$ such that
	\[\langle f,\mu'\rangle>1\ge \langle h,\mu'\rangle\quad\mbox{for all } h\in H.\]
	This shows that $\mu'\in H^\circ$, and therefore $f\not\in H^{\circ\circ}$.
\end{proof}

\begin{corollary}\label{cor:verschieben}
	For every monotone convex regular tight set $H\subset\mathcal{L}_{b-}^0$ which is closed under $\liminf$ and $0\in H$,
	one has $H=H^{\circ\circ}$.
\end{corollary}
\begin{proof}
	If $H=\mathcal{L}^0_{b-}$ the statement obviously holds. Otherwise, $H$ is nontrivial and 
	by Lemma \ref{lem:centered} there exists $m\in\mathbb{R}$ such that $\tilde H:=H-m$ is centered. 
	Following the arguments in the proof of Proposition \ref{thm:main.tight.liminf}, 
	it follows that $f\in \tilde H$ if and only if $\langle f,\mu\rangle\le \alpha_{\tilde H}(\mu)$ 
	for all $\mu\in ca_+^1$, where $\alpha_{\tilde H}(\mu)=\sup_{f\in \tilde H\cap C_b}\langle f,\mu\rangle$ 
	is the support function of $\tilde H\cap C_b$. 
	Hence, $f\in H$ if and only if 
	$f-m\in\tilde H$ if and only if $\langle f,\mu\rangle\le \alpha_{\tilde H}(\mu)+m=\alpha_{H}(\mu)$ for all $\mu\in ca_+^1$.
	Since $0\in H$ we can apply the scaling argument in the end of the proof of Proposition \ref{thm:main.tight.liminf}, which implies
	$H=H^{\circ\circ}$.
\end{proof}

\section{The proof of Theorem \ref{thm:main}}\label{sec:proof}

Throughout this section we assume that $\Omega$ is a $\sigma$-compact metric space, that is, 
there exists a sequence $(K_n)$ of compact subsets of $\Omega$ such that $\Omega=\bigcup_n K_n$. 
For $H\subset \mathcal{L}^0_+$ the bipolar 
$H^{\circ\circ}=\left\{f \in \mathcal{L}_{+}^0: \langle f,\mu\rangle\leq 1\mbox{ for all }\mu\in H^\circ\right\}$ 
is a subset of $\mathcal{L}^0_+$, while for $G\subset \mathcal{L}^0_{b-}$ the bipolar 
$G^{\circ\circ}=\left\{f \in \mathcal{L}_{b-}^0: \langle f,\mu\rangle\leq 1\mbox{ for all }\mu\in G^\circ\right\}$ 
is a subset of $\mathcal{L}^0_{b-}$. 
Recall that $H^\circ=\left\{\mu \in ca_+: \langle f,\mu\rangle\le 1\mbox{ for all }f \in H\right\}$.
In the following we provide the proof of the main result.

\begin{proof}[Proof of  Theorem \ref{thm:main}]
	If $H=H^{\circ\circ}$ then $H$ is convex and closed under $\liminf$ by Fatou's lemma.
	As for the other implication we can assume that $H\neq \mathcal{L}^0_+$
	because otherwise $H=H^{\circ\circ}$ obviously holds. 
	Let $(K_n)$ be an increasing sequence of compact subsets of $\Omega$ such that 
	$\Omega=\bigcup_n K_n$, and define the function $\gamma\colon\Omega\to[0,+\infty)$
	by $\gamma:=\sum_n 1_{K_n^c}$. Then, for every $c\in\mathbb{R}_+$,
	the level set $\{\gamma\leq c\}$ is compact.
	We claim that $H_k$ is nontrivial for $k$ large enough, 
	where $H_k$ is given by 
	\[H_k:=\{ f \in \mathcal{L}_{b-}^0 : f\leq h +\gamma/k\text{ for some } h\in H\}\]
	for all $k\in\mathbb{N}$.
	Indeed, if $H_k$ is trivial for every $k$, then there exists $h_k\in H$ such that
	$k\leq h_k+\gamma/k$ for all $k$. 
	However, since $H$ is closed under $\liminf$, 
	this implies that $h=\liminf_k h_k= +\infty\in H$, in contradiction to  $H\neq \mathcal{L}^0_+$.
    Further, $H_k$ is closed under $\liminf$ for each $k$, since for every sequence $(f_n)$ in $H_k$ with 
    $c\le f_n\le h_n+\gamma/k$ for $h_n\in H$ and $c\in\mathbb{R}$, one has $\liminf_n f_n\le h+\gamma/k$ for $h=\liminf_n h_n\in H$. 
    By Lemma \ref{lem:centered} it follows that $H_k-m_k$ is normalized for a unique $m_k\in\mathbb{R}$.
	In particular, there exists $h_k\in H$ such that $m_k\leq h_k+\gamma/k$.
	For $n\in\mathbb{N}$ define the compact set
	$K:=\{\gamma\leq k(m_k+n)\}$. Then 
	\[m_k+n1_{K^c} \leq h_k+\gamma/k\]
	so that $n1_{K^c}\in H_k-m_k$, that is, $H_k-m_k$ is tight. 
	Since $\gamma$ is lower semicontinuous, there exists a sequence
	$(\gamma_n)$ of continuous bounded functions such that
	$0\leq\gamma_n\uparrow\gamma/k$. For every $\mu\in ca_+^1$ one has
	\begin{align}
	\label{eq:support.non.tight}
		&\sup_{f\in H_k\cap C_b} \langle f,\mu\rangle
		\geq \sup_{f\in H\cap C_b}\sup_{n\in\mathbb{N}} \langle f+\gamma_n,\mu\rangle
		=\sup_{f\in H\cap C_b} \langle f,\mu\rangle + \langle \gamma/k,\mu\rangle\\
		&=\sup_{f\in H\cap U_b} \langle f,\mu\rangle+ \langle \gamma/k,\mu\rangle
		\geq \sup_{f \in H_k\cap U_b}\langle f,\mu\rangle
	\nonumber
	\end{align}
	where we have to justify the last inequality. To that end, fix $f\in H_k\cap U_b$ so that $f\leq h+\gamma/k$ for some $h\in H$.
	Then $0\vee(f-\gamma/k)\leq h$, so that  $0\vee(f-\gamma/k)\in H\cap U_b$ by monotonicity of $H$, 
	which implies that $\langle f,\mu\rangle \leq \langle 0\vee(f-\gamma/k),\mu\rangle + \langle \gamma/k,\mu\rangle$.
	Since $H_k\cap C_b\subset H_k\cap U_b$, it follows from \eqref{eq:support.non.tight} that $H_k$ is regular. 

	In summary, $H_k-m_k$ is a monotone convex normalized regular tight subset of $\mathcal{L}^0_{b-}$. 
	By Corollary \ref{cor:verschieben} one has $H_k=H_k^{\circ\circ}$. Since $\gamma$ is positive, it follows that $H\subset H_k\cap\mathcal{L}_+^0$ for every $k$.
	On the other hand, if $f\in H_k\cap \mathcal{L}_+^0$ for every $k$, then there exists a sequence $(h_k)$ in $H$ such that $f\leq h_k+\gamma/k$.
	But then $f\leq h:=\liminf_k h_k$, and since $H$ is monotone and closed under 
	$\liminf$, it follows that $f\in H$. Thus, one has
	\[	H=\bigcap_k \left(H_k\cap \mathcal{L}_+^0\right)
	=\bigcap_k \left(H_k^{\circ\circ} \cap \mathcal{L}_+^0\right)
	=\Big( \bigcup_k H_k^\circ\Big)^\circ\cap \mathcal{L}_+^0.\]
	Since $\bigcup_k H_k^\circ\subset H^\circ$, it follows that
	$H=\left(\bigcup_k H_k^\circ\right)^\circ\cap \mathcal{L}_+^0 \supset H^{\circ\circ}$. 
	On the other hand $H\subset H^{\circ\circ}$ always holds, so that $H=H^{\circ\circ}$ and the proof is complete.
\end{proof}

\begin{proof}[Proof of Corollary \ref{cor:main}]
	By definition it holds $H^\circ\subset (H\cap C_b)^\circ$. As for the other inclusion, fix
	$\mu\in ca_+$ such that $\mu\notin H^\circ$. Then, by definition, there exists $h\in H$ such
	that $\langle h,\mu\rangle >1$. Further we may assume that $h$ is bounded
	since $h\wedge n\leq h$, and monotonicity of $H$ implies that $h\wedge n\in H$ for every $n$.
	Moreover the measure $\mu$ is tight since $\mu(K_n^c)\downarrow\mu(\emptyset)=0$, and therefore inner regular.
	In particular there exists a sequence $(h_n)$ of upper semicontinuous function such that 
	$0\leq h_n\leq h$ and $\langle h_n,\mu\rangle \to \langle h,\mu\rangle$. 
	Hence $\langle h_{n_0},\mu\rangle>1$ which implies that 
	$\mu\notin (H\cap U_b)^\circ=(H\cap C_b)^\circ$, where the last equality
	holds by assumption. Thus indeed $H^\circ= (H\cap C_b)^\circ$.
\end{proof}

\begin{proof}[Proof of Corollary \ref{cor:functional}]
	On $\mathcal{L}^\infty$ consider the functional $\hat{\phi}(f):=\phi(f\vee 0)$ so that
	$\phi(f)=\sup_n\hat{\phi}(f\wedge n)$ for every $f\in\mathcal{L}_+^0$, see Remark \ref{rem:fatou.continuous.below}.
	By Theorem \ref{thm:main} it holds $\{\phi\leq c\}=\{\phi\leq c\}^{\circ\circ}$ for every $c\in\mathbb{R}$,
	which implies that $\hat{\phi}$ is $\sigma(\mathcal{L}^\infty,ca)$-lower semicontinuous.
	Further one has $\phi^\ast(\mu)=\sup_{f\in\mathcal{L}^\infty}(\langle f,\mu\rangle-\hat{\phi}(f))$.
	Using that $\hat{\phi}$ is increasing, it follows from the Fenchel-Moreau theorem that
	\[ \hat{\phi}(f)=\sup_{\mu\in ca_+} (\langle f,\mu\rangle -\phi^\ast(\mu))  \quad\text{for all }f\in\mathcal{L}^\infty.\]
	The claim then follows by the same arguments as in the proof of Proposition \ref{thm:main.tight.liminf}.
\end{proof}

\section{A transport duality with non-tight marginals}\label{sec:app1}

Let $\Omega_1$ and $\Omega_2$ be two $\sigma$-compact metric spaces and fix two nonempty
monotone and convex sets $H_i\subset\mathcal{L}^0_+(\Omega_i)$. 
For $i=1,2$, we assume that $H_i$ is regular and closed under $\liminf$. It follows from
Theorem \ref{thm:main} that
\begin{align*}
	H_i&=\{ f\in\mathcal{L}_+^0(\Omega_i) : \langle f,\mu\rangle \leq 1 \text{ for } \mu\in H_i^{\circ} \}
	=\{ f\in\mathcal{L}_+^0(\Omega_i) : \pi_i(f)\leq 1 \}
\end{align*}
where the functional $\pi_i:\mathcal{L}_+^0(\Omega_i)\to[0,+\infty]$ is given by 
\[\pi_i(f):=\sup_{\mu\in H_i^\circ} \langle f,\mu\rangle.\]
The space $\Omega:=\Omega_1\times\Omega_2$ is endowed with the product topology. For $h_i\in H_i$, $i=1,2$,
we write $h_1\oplus h_2:\Omega\to[0,+\infty]$ for the function 
$h_1\oplus h_2(\omega):=h_1(\omega_1)+h_2(\omega_2)$.
Define the set 
\[ H:=\{ f\in\mathcal{L}_+^0(\Omega) : f\leq h_1\oplus h_2 \text{ for } h_i\in \mathcal{L}^0_+(\Omega_i)\text{ with } \pi_1(h_1)+\pi_2(h_2)\leq 1 \}. \]
For a measure $\mu\in ca_+(\Omega)$ denote by $\mu_1:=\mu(\cdot \times\Omega_2)\in ca_+(\Omega_1)$ 
and $\mu_2:=\mu(\Omega_1\times\cdot)\in ca_+(\Omega_2)$ its marginal distributions.

\begin{theorem}
	\label{prop:transport}
	Suppose there exist measures $\mu_i^\ast\in ca_+(\Omega_i)$ 
		such that $\mu_i\ll \mu_i^\ast$ for all $\mu_i\in H^\circ_i$, for $i=1,2$.
	Then one has
	\[ H =\{ f\in\mathcal{L}_+^0(\Omega) : \pi(f)\leq 1 \}\]
	where the functional $\pi:\mathcal{L}_+^0(\Omega)\to[0,+\infty]$ is given by
	\begin{equation}\label{gen:transport}
	\pi(f):=\sup\{ \langle f,\mu\rangle : \mu\in ca_+(\Omega) \text{ such that } 
	\mu_1\in H_1^\circ, \mu_2\in H_2^\circ\}. \end{equation}
	In particular, it holds $H=H^{\circ\circ}$, where 
	$H^\circ = \{ \mu\in ca_+(\Omega) :\mu_1\in H_1^\circ, \mu_2\in H_2^\circ\}$.
\end{theorem}

\begin{remarks}
	{\rule{0mm}{1mm}\\[-3.25ex]\rule{0mm}{1mm}}
	\begin{enumerate}[ 1.]
	\item 
	If each $H^\circ_i$ consists of exactly one probability measure $\nu_i$, 
	the optimization problem \eqref{gen:transport} reduces to the Monge-Kantorovich transport problem, 
	see the original paper \cite{kantorovich2006problem} or \cite{villani2008optimal} for a modern view and many applications.
	In case that the functionals $\pi_i$ are linear, the programming duality for measurable functions 
	was first shown in \cite{kellerer1984duality}, see also \cite{beiglbock2011duality}.
	Duality for measurable functions is important as it e.g.~allows to characterize all negligible sets: A Borel set $N\subset \Omega_1\times\Omega_2$ is a $\mu$-zero set for all measures $\mu$
	on the product space with marginals $\mu_i=\nu_i$ if and only if 
	$N\subset (N_1\times\Omega_2)\cup(\Omega_1\times N_2)$ for $\nu_i$-zero sets $N_i\subset\Omega_i$.
	These type of results have geometric applications, see e.g.~\cite{beiglboeck2017optimal} and the discussion therein.
	\item 
	It follows from Theorem \ref{prop:transport} that for each  $f\in\mathcal{L}_+^0(\Omega)$ 
	the following superhedging duality holds:
	\[	\phi(f)
	:=\inf\Big\{m\in\mathbb{R}_+ : f\le f_1\oplus f_2 \text{ for } f_i\in\mathcal{L}^0_+(\Omega_i)\mbox{ with }\sum_i \pi_i(f_i)\le m\Big\}
	=\pi(f).\]
	Indeed, for every $m>0$, one has $f/m\in H$
	if and only if $f\le f_1\oplus f_2$ for $f_i\in\mathcal{L}^0_+(\Omega_i)$ with $\sum_i \pi_i(f_i)\le m$,
	so that $\phi(f)\le m$ whenever $\pi(f)\le m$. In financial terms, the seller of a contingent claim $f$ 
	protects himself against losses by optimally investing in the traded derivatives $f_1$ and $f_2$ with ask prices $\pi_i(f_i)$.
	While for the Kantorovich transport problem the pricing functionals  $\pi_i(f_i)=\langle f_i,\nu_i\rangle$ 
	are linear, in \cite{bartl2017duality} the pricing rules $\pi_i$, $i=1,2$, 
	are assumed to be sublinear reflecting market incompleteness. However, the pricing rules in 
	\cite{bartl2017duality} are continuous from above, i.e.~the marginals $H^\circ_i$ are assumed to be tight.
	\item If both $H_i$ are such that $\mathop{\mathrm{lim\,med}}_n h^n_i\in H$ for every 
	sequence $(h^n_i)$ in $H_i$, then the assumption that $H_i^\circ$ are dominated is not needed 
	(for the concept of medial limits we refer to the next section).
	In that case, one replaces $h_i:=\limsup_n \tilde h^n_i$ by 
	$h_i:=\mathop{\rm lim\,med}_n h^n_i$ in the following proof.
\end{enumerate}
\end{remarks}

\begin{proof}
	The goal is to apply Theorem \ref{thm:main}. It is clear that $H$ is nonempty, monotone, and convex.
	Moreover, for every $\mu\in ca_+(\Omega)$ one has
	\begin{align}
	\label{eq:regular.transport}
	\sup_{h\in H\cap C_b(\Omega)} \langle h,\mu\rangle
	=\max_{i=1,2} \sup_{h_i\in H_i} \langle h_i,\mu_i\rangle,
	\end{align} 
	and the same holds true if $H\cap C_b(\Omega)$ is replaced by $H$. Indeed, since every $h\in H$ satisfies 
	$h\leq h_1\oplus h_2$ for $h_i\in \mathcal{L}^0_+(\Omega_i)$, $i=1,2$, with $\pi_1(h_1)+\pi_2(h_2)\leq 1$,
	it follows that
	\[ \langle h,\mu\rangle 
	\leq \langle h_1\oplus h_2,\mu\rangle
	=\sum_i \langle h_i,\mu_i\rangle \le \sum_i \pi_i(h_i)\sup_{f\in H_i}\langle f,\mu_i\rangle
	\leq \max_{i} \sup_{f\in H_i} \langle f,\mu_i\rangle,\]
	because $h_i/\pi_i(h_i)\in H_i$ so that $\langle h_i,\mu_i\rangle \leq \pi_i(h_i)\sup_{f\in H_i}\langle f,\mu_i\rangle$
	(with the convention $0\cdot(+\infty)=+\infty$). 
	This shows that the right hand side of \eqref{eq:regular.transport} is greater than the left hand side.
	As for the other inequality, assume without loss of generality that the
	maximum on the right hand side is attained at $i=1$.
	By Corollary \ref{cor:main} one has
	 $\sup_{f\in H_1}\langle f,\mu_1\rangle=\sup_{f\in H_1\cap C_b(\Omega_1)}\langle f,\mu_1\rangle$, 
    so that for every $\varepsilon>0$ there exists $h_1\in H_1\cap C_b(\Omega_1)$
	which satisfies $\max_{i} \sup_{f\in H_i} \langle f,\mu_i\rangle\le \langle h_1,\mu_1\rangle+\varepsilon$.
	Define the function $h\in H\cap C_b(\Omega)$ by $h(\omega):=h_1(\omega_1)$.
	Then, since $\langle h,\mu\rangle =\langle h_1,\mu_1\rangle$, it follows that 
    $\max_{i} \sup_{h_i\in H_i} \langle h_i,\mu_i\rangle\le \sup_{h\in H\cap C_b(\Omega)} \langle h,\mu\rangle$. In particular, one has
	\[ H^\circ = \{ \mu\in ca_+(\Omega) \text{ such that } \mu_1\in H_1^\circ, \mu_2\in H_2^\circ\}. \]

	We are left to show that $H$ is closed under $\liminf$. Fix an increasing sequence $(h^n)$ in $H$.
	Then $h^n\leq h_1^n\oplus h_2^n$ for $h_i^n\in \mathcal{L}^0_+(\Omega_i)$ with $\pi_1(h_1^n)+\pi_2(h_2^n)\leq 1$.
	Since $h_i^n\geq0$, 
	we can apply the Koml\'os' theorem (see \cite[Lemma A.1]{delbaen1994general}) to obtain forward convex 
	combinations $\tilde{h}^n_i\in\mathop{\mathrm{conv}}\{h_i^k: k \geq n\}$
	which have a $\mu_i^\ast$-almost sure limit. Define $h_i:=\limsup_n \tilde h^n_i\in \mathcal{L}^0_+$, so that
	$\mu^\ast_i(h_i=\liminf_n \tilde h^n_i)=1$. By the bipolar representation of $H_i$ and Fatou's lemma, it follows that $h_i\in H_i$.
	Moreover, we obtain	\[\pi_1(h_1)+\pi_2(h_2)
	\leq \liminf_n \big( \pi_1(\tilde{h}^n_1)+\pi_2(\tilde{h}^n_2) \big)
	\leq \liminf_n \big( \pi_1(h^n_1)+\pi_2(h^n_2) \big)
	\leq 1 \]
	again by Fatou's lemma and convexity of $\pi_i$.
	But then
	\[ \sup_n h^n=
	\liminf_n \tilde{h}^n 
	\leq \liminf_n (\tilde{h}^n_1\oplus \tilde{h}^n_2)
	\leq \limsup_n (\tilde{h}^n_1\oplus \tilde{h}^n_2)
	\leq h_1\oplus h_2,  \]
	which shows that $h\in H$.
	Therefore, 
	\[ H=\{ h\in\mathcal{L}^0_+(\Omega) : \langle h,\mu\rangle \leq 1 \text{ for }\mu\in H^\circ \}
	=\{ h\in\mathcal{L}^0_+(\Omega) : \pi(f)\leq 1 \}, \]
	where the first equality follows from Theorem \ref{thm:main}.
\end{proof}

\section{Robust hedging in discrete time}\label{sec:app2}

Given a time horizon $T\in\mathbb{N}$, we consider the state space $\Omega:=\mathbb{R}_{++}^T:=(0,+\infty)^T$, and denote by
$S_t\colon\Omega\to\mathbb{R}_{++}$ the projection on the $t$-th coordinate $S_t(\omega)=\omega_t$. We assume that
the canonical process $(S_t)_{t=1,\dots,T}$ describes the discounted price process of a financial asset.
We consider an agent who is allowed to invest dynamically in this asset and statically in a plain vanilla option on $S_T$.
Thus, the set of trading strategies $\Theta$ consists of pairs $(\vartheta,g)$ where $\vartheta=(\vartheta_2,\dots,\vartheta_T)$ and 
each $\vartheta_t\colon\mathbb{R}_{++}^{t-1}\to\mathbb{R}$ is universally measurable, 
$g\colon\mathbb{R}_{++}\to\mathbb{R}\cup\{+\infty\}$ is a Borel measurable function which is bounded from 
below and satisfies $\varphi(g)\leq 0$. Here $\varphi(g)$
denotes the price of the plain vanilla option $g(S_T)$, given by the pricing functional 
\[\varphi(g):=\sup_{\mu\in\mathcal{Q}}\langle g,\mu\rangle,\] 
where  $\mathcal{Q}$ is a set of probability measures on $\mathbb{R}_{++}$. We assume that
$\mathcal{Q}$ is nonempty, convex, and compact w.r.t.~the weak topology induced by the 
continuous bounded functions on $\mathbb{R}_{++}$. The outcome of the trading 
strategy $(\vartheta, g)\in\Theta$ is the universally measurable function 
$(\vartheta\cdot S)_T+g(S_T):\Omega\to\mathbb{R}\cup\{+\infty\}$, where  $g(S_T)(\omega):=g(S_T(\omega))$  and
\[(\vartheta\cdot S)_T(\omega):=\sum_{t=2}^{T}\vartheta_t\left(S_1(\omega),\dots,S_{t-1}(\omega)\right)(S_t(\omega)-S_{t-1}(\omega)).\]
As already mentioned in the introduction, this setting is a discrete-time analogue to \cite{dolinsky2014martingale}, however we allow for sublinear pricing functionals $\varphi$. Related results are given in \cite{bouchard2015arbitrage} where (for linear pricing functionals) the duality and the existence of optimal strategies are shown be means of dynamic programming. In martingale optimal transport, where static options are available for all maturities $t=1,\dots,T$,
the duality is shown in \cite{beiglbock2013model}
for semicontinuous functions $f$ (see also \cite{acciaio2016model}), and in \cite{beiglbock2017complete} for measurable $f$ under the assumption that $T=2$. See also the recent book \cite{henry2017model} for an overview.

In the following we make use of so-called medial limits, see \cite{meyer1973limites,normann1976martin}. A medial limit is a positive linear functional 
$\mathop{\rm lim\,med}\colon l^\infty\to\mathbb{R}$ which satisfies $\liminf\le\mathop{\rm lim\,med}\le\limsup$
and
$\omega\mapsto f(\omega):=\mathop{\rm lim\,med}_n f_n(\omega)$ is universally measurable
for every bounded sequence of universally measurable functions $(f_n)$. We assume that a medial limit exists,
which for instance is guaranteed under the usual axioms of ZFC and Martin's Axiom.
For a discussion of the medial limit as a tool for pointwise convex optimization problems we refer to 
\cite{bartl2017robust}, and as a tool for the aggregation of stochastic integrals to \cite{nutz2012pathwise}.

\begin{proposition}
	\label{prop:semistatic}
	Assume that $\lim_{k\to\infty}\varphi((id-k)\vee 0)=0$, there exists $\mu^\ast\in\mathcal{Q}$ such that 
	$\mu\ll\mu^\ast$ for all $\mu\in\mathcal{Q}$,
	and the smallest interval containing the support of $\mu^\ast$ equals $\mathbb{R}_{++}$. Then one has
	\begin{align}
	&\{ f\in \mathcal{L}_{b-}^0: f\leq (\vartheta\cdot S)_T + g(S_T) \text{ for some } (\vartheta,g)\in \Theta \}\nonumber\\
	=&\{ f\in \mathcal{L}_{b-}^0 : \langle f,\mu\rangle \leq 0 \text{ for all } \mu\in\mathcal{M}(\mathcal{Q}) \}\label{attainable}
	\end{align} 
	where $\mathcal{M}(\mathcal{Q})$ denotes the set of all martingale measures $\mu$ for $S$ which satisfy
	$\mu_T:=\mu\circ S_{T}^{-1}\in\mathcal{Q}$.
\end{proposition}

\begin{remarks}{\rule{0mm}{1mm}\\[-3.25ex]\rule{0mm}{1mm}}
	\begin{enumerate}[ 1.]
	\item It follows from equation \eqref{attainable} that the set of all bounded attainable outcomes
	\[\{ f\in \mathcal{L}^\infty: f\leq (\vartheta\cdot S)_T + g(S_T) \text{ for some } (\vartheta,g)\in \Theta \}\]
	is $\sigma(\mathcal{L}^\infty,ca)$-closed. In general, the set of attainable outcomes 
	under semistatic hedging is not closed, see \cite{acciaio2017space}. 
	Moreover, equation \eqref{attainable} implies that for each $f\in\mathcal{L}^0_{b-}$ and $m\in\mathbb{R}$ one has
	$f\leq m+(\vartheta\cdot S)_T+g(S_T)$ for some $(\vartheta,g)\in\Theta$ 
	if and only if $\langle f,\mu\rangle\le m$ for all $\mathcal{M}(\mathcal{Q})$,
	which yields the superhedging duality
	 \[\inf\left\{ m\in\mathbb{R}: m+(\vartheta\cdot S)_T + g(S_T)\ge f \text{ for some } (\vartheta,g)\in \Theta \right\}
	 =\sup_{\mu\in\mathcal{M}(\mathcal{Q})}\langle f,\mu\rangle.\]

	\item Even though $\mathcal{Q}$ is dominated by the probability measure $\mu^\ast$,
	one can check that the set of pricing measures $\mathcal{M}(\mathcal{Q})$ is not dominated in general. 

	\item If every $\mu\in\mathcal{Q}$ has the same barycenter $S_0=\langle id,\mu\rangle\in\mathbb{R}_{++}$, 
	then Proposition \ref{prop:semistatic} holds for 
	extended trading strategies $(\vartheta_1,\vartheta,g)$ with $\vartheta_1\in\mathbb{R}$ and $(\vartheta,g)\in\Theta$.
		
	\item If instead of the state space $\Omega=\mathbb{R}_{++}^T$ one considers $\Omega=[0,+\infty)^T$, 
	Proposition \ref{prop:semistatic} does not hold unless one allows $\vartheta_t$ to assume the value $+\infty$.
	To see this, let $T=2$ and $\mathcal{Q}$ be the convex hull
	of $\mu^\ast$ and $\{ \xi_nd\lambda :n\in\mathbb{N}\}$, where $\mu^\ast:=(\delta_0+\xi d\lambda)/2$
	for a strictly positive density $\xi$ (w.r.t.~the Lebesgue measure $\lambda$) 
	with finite first moment and $\xi_n d\lambda \to d\mu^\ast$. It is possible to choose $(\xi_n)$ such that
	  $\mathcal{Q}$ fulfills the assumptions of Proposition \ref{prop:semistatic}. Define $f:=1_{\{0\}\times(0,1)}$ so that 
	$\langle f,\mu\rangle=0$ for every $\mu\in\mathcal{M}(\mathcal{Q})$,
	and let $(\vartheta,g)$ such that $(\vartheta\cdot S)_T+g(S_T)\geq f$.
	We will see in the proof of the proposition that $g$ has to be positive.
	Hence, whenever $\vartheta_2(0)\neq+\infty$ it follows that
	$g(x)\geq (1- \vartheta_2(0)x)\vee 0$ for $x\in(0,1)$ and therefore
	$\varphi(g)\geq \sup_n \int_0^1 ((1-\vartheta_2(0)x)\vee 0)\xi_n(x)\, dx\geq 1/2$.
	\end{enumerate}
\end{remarks}

\begin{proof}[Proof of Proposition \ref{prop:semistatic}]
	The goal is to apply Proposition \ref{thm:main.tight.liminf} to the set
	\[ H:=\{ h\in\mathcal{L}_{b-}^0
	: h\leq (\vartheta\cdot S)_T + g(S_T) \text{ for some } (\vartheta,g)\in \Theta \}.\]
	It is clear that $H$ is monotone and contains 0, and we claim that $H$ is normalized and tight. 
	If $m\in H$ for some $m\geq 0$, then
	$m\leq g(x)$ for every $x\in\mathbb{R}_{++}$ since $(\vartheta\cdot S)_T=0$
	on the constant path $\omega=(x,\dots,x)$. Since $\varphi(g)\leq 0$ it follows that $m=0$.
	To show that $H$ is tight, fix $\varepsilon>0$ and $n\in\mathbb{N}$. 
	Due to compactness of the set $\mathcal{Q}$, one can show that there exist $\delta>0$ such that
	$\mu((0,2\delta])\leq \varepsilon/(2n)$ for every $\mu\in\mathcal{Q}$.
	In combination with the assumption that $\lim_k\varphi((id-k)\vee 0)=0$, there thus exists $k\in\mathbb{N}$
	such that $\varphi(g)\leq0$ where $g(x):=nx1_{[k-1,\infty)}(x)+2n1_{(0,2\delta]}(x)-\varepsilon$.
	Define the stopping times $\tau:=\inf\{ t\ge 1: S_t>k\}$  and $\sigma:=\inf\{t\ge 1 : S_t<\delta\}$
	as well as $\vartheta_t:=-n1_{\{t\geq \tau +1\}}+n/\delta1_{\{t\geq\sigma+1\}}$.
	Then $(\vartheta\cdot S)_T+g(S_T)\geq m1_{K^c}-\varepsilon$ for $K:=[\delta,k]^T$ so that
	$H$ is tight.
	
	We next show that $\sup_n h_n\in H$ whenever $(h_n)$ is an increasing
	sequence in $H$. Since $h_1\in \mathcal{L}^0_{b-}$ there exists $c\in\mathbb{R}$ with
	$h_n\geq h_1\geq c$.
	Let	$(\vartheta^n, g^n)\in\Theta$ such that $(\vartheta^n\cdot S)_T+g^n(S_T)\geq h_n$.
	Considering the constant path $\omega=(x,\dots,x)$ it follows that
	$c\leq h_n(x,\dots,x)\leq g^n(x)$. Since $\langle g^n,\mu^\ast\rangle\leq \varphi(g^n)\leq 0$,
	we can apply the Koml\'os' theorem (see \cite[Lemma A.1]{delbaen1994general})
	in order to obtain a sequence of forward convex combinations
	$\tilde{g}^n\in\mathrm{conv}\{g^k:k\ge n\}$
	which converge $\mu^\ast$-almost surly to a Borel measurable function $g:\mathbb{R}_{++}\to[c,+\infty]$.
	By convexity of $\varphi$ it holds $\varphi(\tilde{g}^n)\leq 0$ and 
	by Fatou's lemma it follows that $\varphi(g)\leq 0$ so that the Borel set
	\[C:=\{ x\in\mathbb{R}_{++} : \tilde{g}^n(x)\to g(x)\in\mathbb{R}\}\]
	has $\mu^\ast$-measure one. Redefine $g$ to be $+\infty$ on the complement of this set.
	Passing to the same convex combinations used for $\tilde{g}_n$ also for $\vartheta^n$ and $h_n$, 
	it holds in obvious notation that $\tilde{g}^n+(\tilde{\vartheta}^n\cdot S)_T\geq\tilde{h}_n$.
	For the purpose of readability we again write $g^n$, $\vartheta^n$ and $h_n$.
	Assume that there exists $x\in\mathbb{R}_{++}$ such that the sequence $(\vartheta_2^n(x))$ is not bounded.
	We focus on the case that $\limsup_n\vartheta_2^n(x)=+\infty$, the other case is treated analogously.
	Since $(\vartheta^n\cdot S)_T\geq c- g^n$,
	it follows for any path of the form $\omega=(x,y,\dots,y)\in\Omega$ with $y\in(0,x)$ that
	\[ -\infty=\liminf_n \big( \vartheta_2^n(x)(y-x) \big) 
	=\liminf_n (\vartheta^n\cdot S)_T(\omega)
	\geq \liminf_n (c-g^n(y)). \]
	However, since $(g^n(y))$ is bounded for $y\in C$ and  $C\cap (0,x)\neq \emptyset$ by assumption, 
	this already yields a contradiction.
	By induction it follows that $(\vartheta_t^n)$ is pointwise bounded for every $t$, 
	so that $\vartheta_t:=\mathop{\mathrm{lim\,med}}_n \vartheta_t^n$ 
	is well-defined.  By monotonicity of the medial limit it follows that
	\[\sup_n h_n
	=\mathop{\mathrm{lim\,med}}_n h_n
	\leq \mathop{\mathrm{lim\,med}}_n \big((\vartheta^n\cdot S)_T+g^n(S_T) \big)
	\leq (\vartheta\cdot S)_T+ g(S_T) \]
	which shows that $\sup_n h_n\in H$.
	
	We finally show that $H$ is regular, that is 
	\[a:=\sup_{h\in H\cap U_b}\langle h,\mu\rangle
	=\sup_{h\in H\cap C_b}\langle h,\mu\rangle=:b\]
	for every $\mu\in ca_+^1$.
	First notice that $a\geq b$ and since $\lambda h\in H$ for every 
	$h\in H$ and $\lambda \in\mathbb{R}_+$, it follows that $a,b\in\{0,+\infty\}$.
	Let $\mu\in\mathcal{M}(\mathcal{Q})$, $h\in H$ (not necessarily upper semicontinuous) and $(\vartheta,g)\in\Theta$ such that
	$h\le (\vartheta\cdot S)_T+g(S_T)$.
	It follows from a result on local martingales that
	$\langle (\vartheta\cdot S)_T,\mu\rangle =0$
	(see \cite[Theorem 1 and Theorem 2]{jacod1998local}). 
	Hence $\langle h,\mu\rangle \leq\varphi(g)\leq 0$, so that $a=b=0$,
	and in particular $\mu\in H^\circ$. Conversely, let $\mu\not\in\mathcal{M}(\mathcal{Q})$.
    First, if $\mu_T\notin\mathcal{Q}$, the hyperplane separation theorem yields 
	the existence of a continuous bounded function $g\colon \mathbb{R}_{++}\to\mathbb{R}$ such that
	$\langle g,\mu_T\rangle>0$ and $\varphi(g)\leq 0$. 
	For $h:=g\circ S_T\in H\cap C_b$ it follows that $b\ge \langle h,\mu\rangle>0$, and therefore $b=+\infty$. 
	Second, if $S_t\notin L^1(\mu)$ for some $t\in\{1,\dots,T\}$, then $b=+\infty$.
	Indeed,
	\[(S_t-\varphi(id))\wedge k\leq (\vartheta\cdot S)_T+g(S_T)\quad \text{for }g(x):=x-\varphi(id)
	\text{ and } \vartheta_s:=-1_{\{s\geq t+1\}}\] 
	implies that $(S_t-\varphi(id))\wedge k\in H\cap C_b$
	and therefore $b\ge \langle (S_t-\varphi(id))\wedge k,\mu\rangle\to +\infty$.
	Third, if $S$ is not a martingale under $\mu$, then there
	exists $t\in\{1,\dots,T\}$ and a continuous function $\xi$ of the first $t-1$ components 
	with values in $[-1,1]$ such that
	$\varepsilon:=\langle \xi\cdot (S_t-S_{t-1}),\mu\rangle> 0$,
	where, by integrability of $S_t-S_{t-1}$, we may assume that $\xi$ has support 
	$(0,k]^{t-1}$. Define 
	\[f:=(-n)\vee \big( \xi\cdot (S_t-S_{t-1}) \big)\wedge n\in C_b,\]
	and the strategy $(\vartheta_s)$ as $0$ if $s<t$, 
	$\xi$ if $s=t$, and $-1_{\{S_t\geq n-k\}}$ if $s\ge t+1$, as well as $g(x):=x1_{[n-k,\infty)}(x)-\varepsilon/2$.
    For $n\in\mathbb{N}$ large enough one has $\varphi(g)\leq 0$, and since 
	\[(\vartheta\cdot S)_T+g(S_T)
	= S_T1_{\{S_T\geq n-k\}} + \xi\cdot (S_t-S_{t-1})+(S_t-S_T)1_{\{S_t\geq n-k\}}-\varepsilon/2\]
	and the fact that $S_t\geq n-k$ whenever $\xi\cdot (S_t-S_{t-1})\leq -n$,
	it follows that $f-\varepsilon/2\leq (\vartheta\cdot S)_T+g(S_T)$ and therefore $\langle f-\varepsilon/2,\mu\rangle>0$, 
	which shows that $b>0$ and consequently $b=+\infty$. 
	Hence, for $\mu\not\in\mathcal{M}(\mathcal{Q})$ it holds $\sup_{h\in H}\langle h,\mu\rangle=+\infty$, 
	so that $\mu\notin H^\circ$. In summary, one has 
	$H^\circ=\{\lambda\mu : \lambda\in\mathbb{R}_+,\,\mu\in\mathcal{M}(\mathcal{Q})\}$.

	In view of Proposition \ref{thm:main.tight.liminf} we conclude that 
    $f\in H$ if and only if $\langle f,\mu\rangle \leq 1$ for all
	$\mu\in H^\circ$ if and only if $\langle f,\mu\rangle \leq 0$ for all
	$\mu\in\mathcal{M}(\mathcal{Q})$.
\end{proof}

\bibliographystyle{abbrv}

\begin{thebibliography}{10}

\bibitem{acciaio2016model}
B.~Acciaio, M.~Beiglb{\"o}ck, F.~Penkner, and W.~Schachermayer.
\newblock A model-free version of the fundamental theorem of asset pricing and
  the super-replication theorem.
\newblock {\em Mathematical Finance}, 26(2):233--251, 2016.

\bibitem{acciaio2017space}
B.~Acciaio, M.~Larsson, and W.~Schachermayer.
\newblock The space of outcomes of semi-static trading strategies need not be
  closed.
\newblock {\em Finance and Stochastics}, 21(3):741--751, 2017.

\bibitem{bartl2017robust}
D.~Bartl, P.~Cheridito, and M.~Kupper.
\newblock Robust expected utility maximization with medial limits.
\newblock {\em arXiv preprint arXiv:1712.07699}, 2017.

\bibitem{bartl2017duality}
D.~Bartl, P.~Cheridito, M.~Kupper, and L.~Tangpi.
\newblock Duality for increasing convex functionals with countably many
  marginal constraints.
\newblock {\em Banach Journal of Mathematical Analysis}, 11(1):72--89, 2017.

\bibitem{beiglboeck2017optimal}
M.~Beiglb{\"o}ck, A.~M. Cox, and M.~Huesmann.
\newblock Optimal transport and skorokhod embedding.
\newblock {\em Inventiones mathematicae}, 208(2):327--400, 2017.

\bibitem{beiglbock2013model}
M.~Beiglb{\"o}ck, P.~Henry-Labord{\`e}re, and F.~Penkner.
\newblock Model--independent bounds for option prices -- a mass transport
  approach.
\newblock {\em Finance and Stochastics}, 17(3):477--501, 2013.

\bibitem{beiglbock2017complete}
M.~Beiglb{\"o}ck, M.~Nutz, and N.~Touzi.
\newblock Complete duality for martingale optimal transport on the line.
\newblock {\em The Annals of Probability}, 45(5):3038--3074, 2017.

\bibitem{beiglbock2011duality}
M.~Beiglb{\"o}ck and W.~Schachermayer.
\newblock Duality for borel measurable cost functions.
\newblock {\em Transactions of the American Mathematical Society},
  363(8):4203--4224, 2011.

\bibitem{bouchard2015arbitrage}
B.~Bouchard and M.~Nutz.
\newblock Arbitrage and duality in nondominated discrete-time models.
\newblock {\em The Annals of Applied Probability}, 25(2):823--859, 2015.

\bibitem{brannath1999bipolar}
W.~Brannath and W.~Schachermayer.
\newblock {A bipolar theorem for subsets of $L^{0}_{+}(\Omega,\mathcal{F},P)$}.
\newblock {\em S{\'e}minaire de Probabilit{\'e}s XXXIII}, 33:349--354, 1999.

\bibitem{choquet1959forme}
G.~Choquet.
\newblock Forme abstraite du th{\'e}or{\`e}me de capacitabilit{\'e}.
\newblock {\em Annales de l'institut Fourier}, 9:83--89, 1959.

\bibitem{delbaen1994general}
F.~Delbaen and W.~Schachermayer.
\newblock A general version of the fundamental theorem of asset pricing.
\newblock {\em Mathematische annalen}, 300(1):463--520, 1994.

\bibitem{dolinsky2014martingale}
Y.~Dolinsky and H.~M. Soner.
\newblock Martingale optimal transport and robust hedging in continuous time.
\newblock {\em Probability Theory and Related Fields}, 160(1-2):391--427, 2014.

\bibitem{galichon2014stochastic}
A.~Galichon, P.~Henry-Labordere, and N.~Touzi.
\newblock A stochastic control approach to no-arbitrage bounds given marginals,
  with an application to lookback options.
\newblock {\em The Annals of Applied Probability}, 24(1):312--336, 2014.

\bibitem{henry2017model}
P.~Henry-Labord{\`e}re.
\newblock {\em Model-free Hedging: A Martingale Optimal Transport Viewpoint}.
\newblock CRC Press, 2017.

\bibitem{hobson2011skorokhod}
D.~Hobson.
\newblock The skorokhod embedding problem and model-independent bounds for
  option prices.
\newblock {\em Paris-Princeton Lectures on Mathematical Finance 2010}, pages
  267--318, 2011.

\bibitem{jacod1998local}
J.~Jacod and A.~Shiryaev.
\newblock Local martingales and the fundamental asset pricing theorems in the
  discrete-time case.
\newblock {\em Finance and Stochastics}, 2(3):259--273, 1998.

\bibitem{kantorovich2006problem}
L.~V. Kantorovich.
\newblock On a problem of monge.
\newblock {\em Journal of Mathematical Sciences}, 133(4):1383--1383, 2006.

\bibitem{kellerer1984duality}
H.~Kellerer.
\newblock Duality theorems for marginal problems.
\newblock {\em Zeitschrift f{\"u}r Wahrscheinlichkeitstheorie und verwandte
  Gebiete}, 67(4):399--432, 1984.

\bibitem{maggis2016fatou}
M.~Maggis, T.~Meyer-Brandis, and G.~Svindland.
\newblock The fatou property under model uncertainty.
\newblock {\em Positivity, forthcomming, arXiv:1610.04085}, 2016.

\bibitem{meyer1973limites}
P.-A. Meyer.
\newblock Limites m{\'e}diales, d'apr{\`e}s {M}okobodzki.
\newblock {\em S{\'e}minaire de Probabilit{\'e}s VII}, pages 198--204, 1973.

\bibitem{normann1976martin}
D.~Normann.
\newblock Martin's axiom and medial functions.
\newblock {\em Mathematica Scandinavica}, 38(1):167--176, 1976.

\bibitem{nutz2012pathwise}
M.~Nutz.
\newblock Pathwise construction of stochastic integrals.
\newblock {\em Electron. Commun. Probab}, 17(24):1--7, 2012.

\bibitem{villani2008optimal}
C.~Villani.
\newblock {\em Optimal transport: old and new}, volume 338.
\newblock Springer Science \& Business Media, 2008.

\end{thebibliography}

\end{document}